\documentclass[a4paper,11pt]{article}

\usepackage[top=25mm,bottom=30mm,hmargin=25mm,paperwidth=210mm,paperheight=11in]{geometry}
\usepackage[utf8x]{inputenc}
\usepackage[T1]{fontenc}
\usepackage{lmodern}
\usepackage[UKenglish]{babel} 

\usepackage[centertags]{amsmath}
\usepackage{amsfonts}
\usepackage{amssymb}
\usepackage{amsthm}
\usepackage{graphicx}
\usepackage[noadjust]{cite}

\usepackage[normalem]{ulem}
\usepackage[hidelinks]{hyperref}
\usepackage[colorinlistoftodos]{todonotes}
\usepackage{authblk}
\usepackage{xcolor}
\usepackage{color}

\usepackage{setspace}
\usepackage{pgfplots}
\pgfplotsset{compat=1.9}
\usetikzlibrary{calc}
\usetikzlibrary{patterns}
\definecolor{myblue}{rgb}{0.1,0.4,1}
\definecolor{mygreen}{rgb}{0,0.6,0}
 \pgfdeclarepatternformonly{north west spaced lines}{\pgfqpoint{-1pt}{-1pt}}{\pgfqpoint{10pt}{10pt}}{\pgfqpoint{9pt}{9pt}}%
{
 \pgfsetlinewidth{0.4pt}
 \pgfpathmoveto{\pgfqpoint{0pt}{10pt}}
 \pgfpathlineto{\pgfqpoint{10.1pt}{-0.1pt}}
 \pgfusepath{stroke}
}

\usepackage[T1]{fontenc} 
\usepackage{tikz,amsmath,amssymb}
\usepackage{mathrsfs}
\usepackage{pgf}
\usetikzlibrary{matrix}
\graphicspath{{figures/}}

\def\g{\gamma}
\def\s{\sigma}

\def\La{{\Lambda}}

\newcommand{\dd}{\text{\rm d}\mkern0.5mu}

\def\R{{\mathbb R}}
\def\Z{{\mathbb Z}}
\def\NN{{\mathbb N}}

\def\E{{\mathbb E}}
\def\P{{\mathbb P}}

\def\one{{\mathbf 1}}

\def\cS{{\mathcal{S}}}

\newcommand{\floor}[1]{\lfloor #1 \rfloor}

\def\N1{{\mathcal{N}_1}}

\newtheorem{teo}{Theorem}[section]
\newtheorem{lema}[teo]{Lemma}

\theoremstyle{definition}
\newtheorem{df}[teo]{Definition}

\theoremstyle{remark}
\newtheorem{obs}[teo]{Remark}

\title{Gibbs measures over permutations \\ of point processes with low density}
\date{}
\author{In{\'e}s Armend{\'a}riz}
\author{Pablo A. Ferrari}
\author{Nicol{\'a}s Frevenza}
\affil{CONICET and Departamento de Matem{\'a}tica, FCEyN,
Universidad de Buenos Aires}

\begin{document}

\maketitle

\begin{abstract}
\noindent We study a model of spatial random permutations over a discrete set of points. Formally, a permutation $\s$ is sampled proportionally to \begin{equation} 
\label{peso}\exp\{-\alpha \sum_x V(\s(x)-x)\},
\end{equation} where $\alpha>0$ is the temperature and $V$ is a non-negative and continuous potential. The most relevant case for physics is when $V(x)=\|x\|^2$, since it is related to Bose-Einstein condensation through a representation introduced by Feynman in 1953. In the context of statistical mechanics, the weights in~\eqref{peso} define a probability when the set of points is finite, but the construction associated to an infinite set is not trivial and may fail without appropriate hypotheses. The first problem is to establish conditions for the existence of such a measure at infinite volume when the set of points is infinite. Once existence is derived, we are interested in establishing its uniqueness and the cycle structure of a typical permutation.

\noindent We here consider the large temperature regime when the set of points is a Poisson point process in $\Z^d$ with intensity $\rho \in(0,1/2)$, and the potential verifies some regularity conditions. In particular, we prove that if $\alpha$ is large enough, for almost every realization of the point process, there exists a unique Gibbs measure that concentrates on finite cycle permutations.

\noindent We then extend these results to the continuous setting, when the set of points is given by a Poisson point process in $\R^d$ with low enough intensity.

\vspace{4pt}


\end{abstract}
%
%


\section{Introduction} 

We consider a model of spatial random permutations. The interest in these permutations was initially driven by their connection to Bose-Einstein condensation. Richard Feynman \cite{Feynman53} introduced a representation of the Bose gas through trajectories of interacting Brownian motions that evolve over a fixed time interval, starting and finishing at the points of a spatial point process. Several simplifications have been proposed over the years to reduce this representation to spatial random permutations. S{\"u}t{\H{o} showed that macroscopic cycles are present in the ideal Bose gas~(\cite{Suto1,Suto2}). T\'oth~\cite{Toth} relates a particular model of spatial random permutations, the interchange process, to the quantum Heisenberg ferromagnetic model, proving that the existence of macroscopic cycles in the former is equivalent to the presence of spontaneous magnetization for the latter.

Let $\Omega \subset \R^d$ be a locally finite set, that is a set of points such that its intersection with any compact subset of $\R^d$ is finite. We focus on the cases when $\Omega$ equals the integer lattice or is given by a realization of a spatial point process. Let the space state $S_{\Omega}$ given by the set of permutations or bijections $\s\colon \Omega \to \Omega$. We want to consider the probability measure $\mu$ formally given by
\begin{equation} 
\label{medida-formal}
\mu(\s) = \frac{e^{-\alpha H(\s)}}{Z}\,, \hspace{1cm} \s\in S_{\Omega}, \, \alpha>0\,,
\end{equation} 
where $Z$ is a normalization factor and $H$ the Hamiltonian, formally defined by 
\begin{equation} 
\label{Hamiltonian-intro-general}
 H(\s)= \sum_{x \in \Omega} \| \s(x) - x \|^2\,, \hspace{1cm} \s\in S_{\Omega} \,.
\end{equation} 
The Hamiltonian discourages the appearance of large jumps, whose probability decays exponentially. The parameter $\alpha$ is physically interpreted as the temperature of the system, we can also understand it as a degree of penalization of large jumps. Note that if $\Omega$ is finite the measure~\eqref{medida-formal} is well-defined. In general, for infinite $\Omega$, the definition~\eqref{medida-formal} does not make sense. Statistical mechanics provides a standard approach to extending finite volume probabilities to infinite volume. The method consists of specifying what the conditional probabilities of the infinite volume measure given boundary conditions outside a compact set should look like, these so-called specifications are given by~\eqref{medida-formal} plus consistency with the boundary conditions, and then proving that the specifications have weak limits, called Gibbs measures. One of the fundamental problems is to determine whether there exists more than one Gibbs measure.

Once existence and uniqueness of Gibbs measures have been established, the next questions are related to the cycle structure of a typical permutation, and how this depends on parameters such as point density and temperature $\alpha$. These issues are usually studied for identity boundary conditions. In the following, whenever we discuss cycle-lengths properties, the reader should assume that the boundary conditions under consideration are given by the identity. 
In particular, it is interesting to know if infinitely long cycles appear with positive probability, and in this case, whether they are macroscopic, i.e., if their intersection with a large finite box contains a positive density of the points in the box. 

For the model determined by the quadratic Hamiltonian~\eqref{Hamiltonian-intro-general} it is conjectured that, at all temperatures and with identity boundary conditions, a typical permutation will decompose into finite cycles if $d=1,\,2$, whereas for $d\geq 3$, there exists a critical value $\alpha_c$ below which a typical permutation contains an infinite cycle with positive probability. This conjecture can be explained heuristically and it is supported by numerical simulations (see for instance~\cite{Gandolfo-Ruiz-Ueltschi,Grosskinsky-Ueltschi,GUW-quantum}).

The first rigorous results were obtained by Gandolfo, Ruiz and Ueltschi~\cite{Gandolfo-Ruiz-Ueltschi} when the set of points is the integer lattice $\Z^d$. They show that for high enough temperatures all cycles are finite, in any dimension. The case $d=1$ was settled by Biskup and Richthammer in~\cite{Biskup-Rich}. They prove uniqueness of the Gibbs measure associated to identity boundary conditions, at any temperature, and show that it is supported on finite cycle permutations.
They also establish a bijection between ground states (local minima) of the Hamiltonian and extremal Gibbs measures. This one-to-one correspondence is expected to fail in dimensions higher than 1. Armend\'ariz, Ferrari, Groisman and Leonardi~\cite{AFGL} consider the large temperature regime for general strictly convex potentials in $d\ge 2$. They derive the existence and uniqueness of Gibbs measures concentrating on finite cycle permutations for large $\alpha$, determined by the potential.

We are interested in the case when the set of points $\Omega$ is random. In the $1$-dimensional case, Biskup and Richthammer~\cite{Biskup-Rich} show that if $\Omega$ is given by a realization of a point process with a translation invariant distribution, then, as when the points belong to the integer lattice, cycles are almost surely finite at all temperatures. This is a quenched result. 

In the annealed case, when points and permutations are jointly sampled, Betz and Ueltschi proved in~\cite{Betz-Ueltschi-CMP, Betz-Ueltschi-PTRF} that if $d\geq 3$ there exists a critical density of points $\rho_c$ below which a typical permutation has only finite cycles, and above $\rho_c$ it contains macroscopic cycles. The asymptotic behavior of these cycle-lengths is derived in~\cite{Betz-Ueltschi-PD-EJP}. Precisely, the sorted lengths of the cycles, scaled down by $N$, converges to the Poisson-Dirichlet distribution, as was previously proved to be the case for uniformly distributed permutations~\cite{Schramm-trasposiciones}.

We here consider the set of points given by a realization of a Poisson process on $\Z^d$ with intensity $\rho$, that is, for each $x\in\Z^d$ we place a Poisson$(\rho)$ number of points $\theta(x)$ at $x$, independently among locations, and consider the set $\Omega_\theta=\{(x,i) \colon 1\leq i\leq \theta(x),\, x\in \Z^d\}$ hence determined. 
This is a simpler version of a regular Poisson point process on $\R^d$ with intensity $\rho$, obtained by collecting all points in the $d$-dimensional hypercube $x+[0,1)^d$ and placing them at $x$. We study the permutation group of $\Omega_{\theta}$ under the probability induced by~\eqref{Hamiltonian-intro-general}. In this model there might be more than one point per site in $\Z^d$, so the norm appearing in each term of~\eqref{Hamiltonian-intro-general} will measure the distance between the projections of $x$ and $\sigma(x)$, $x\in \Z^d$. In particular, if $x$ and $\sigma(x)$ project onto the same point, then the term indexed by $x$ vanishes. Our first result proves the existence of Gibbs measures for almost every realization of the environment $\{\theta(x)\}_{x\in \Z^d}$, in the large temperature regime and with fixed density $\rho\in (0,1/2)$. We next show that any permutation sampled with respect to this Gibbs measure has finite cycles almost surely with respect to the Poisson point process. We finally derive uniqueness for Gibbs measures supported on the set of finite cycle permutations of $\Omega_\theta$, for almost all $\{\theta(x)\}_{x\in \Z^d}$. 

The results for the discrete quenched case can be extended to the continuous setting when the set of points is a realization of a Poisson point process on $\R^d$ with low enough intensity. Precisely, we show that if the intensity $\rho$ is small enough, in the large temperature regime, there exists a Gibbs measure for almost every realization of the point process, which is unique with the property of being supported on permutations with finite cycle decomposition.

To prove these results we follow the approach introduced by Fern\'andez, Ferrari and Garcia in~\cite{FFG} to study the Peierls contours of the low temperature Ising model as a loss network. Applied to our setting, this method attempts to realize the Gibbs measure as the stationary distribution of an interacting birth and death process Markov process on the space of finite cycles, where a new cycle can only be added  to the configuration if no site in its support is already visited by an existing cycle. This construction requires that the set of cycles determining whether a given cycle $\gamma$ can be born at time $t$, which we call the clan of ancestors, be almost surely finite. To ensure this condition 
the authors in~\cite{FFG,AFGL} introduce a family of subcritical multi-type branching processes that dominate the size of the clan. In the particular case considered here, however, the 
fact that the point process distributes an unbounded number of points in each position of $\Z^d$ implies that we cannot apply the method directly.
Instead, we use the loss network in bounded regions to prove that the specifications on these regions are dominated by a Poisson point process on the space of finite cycles, from where we derive the tightness of the specifications. This proves the existence of Gibbs measures supported on finite cycle permutations, and we also show that there is uniqueness in this class. 

In the next section we introduce the Gibbs formalism for the model and state our main results. In  \S\ref{section-domination} we explain the loss network construction. The results pertaining to the discrete setting are proved in \S\ref{section-existence} and \S\ref{section-uniqueness},  while in \S\ref{section-continuum} we consider the continuous case when the points are in $\R^d$.


\section{Setting and results}

Let $\theta=\{\theta(x)\}_{x\in \Z^d}$ be an i.i.d. sequence of Poisson random variables with mean $\rho$. We will call $\theta(x)$ the multiplicity of site $x \in \Z^d$, and denote by $\P$ and $\E$ the probability and expectation associated to the distribution of $\theta$. 

Given a fixed realization of $\theta$, define
\begin{equation*}
\Omega_{\theta} = \{(x,i)\in \Z^d\times \mathbb{N} \colon \theta(x)>0,\,i=1,\dots,\theta(x)\}, 
\end{equation*}
and let $X: \Omega_\theta\to \Z^d$ be the projection mapping,
\begin{equation*}
X(s)=x\quad\text{if}\quad s=(x,i),\, i\le \theta(x).
\end{equation*}
We will abuse notation and say that a point $s\in \Omega_\theta$ belongs to $\Lambda\subset \Z^d$, $s \in \Lambda$, if $X(s)\in \Lambda$. We write $\Lambda \Subset \Z^d$ to denote that $\La$ is a finite set. In general, we will use the term {\it point} to refer to an element of $\Omega_{\theta}$ and {\it site} for its location in $\Z^d$.

Denote by 
\begin{equation*}
S_\theta:=\{\sigma:\Omega_\theta \to \Omega_\theta \text{ bijective} \}
\end{equation*}
with the topology determined by the distance
\begin{equation*}
d(\s, \s')= 2^{-\min\{\|X(s)\| \colon \s(s)\neq \s'(s)\}}\,,
\end{equation*}
$\min \emptyset= -\infty$, $\|\cdot\|$ the Euclidean norm in $\R^d$. The space $(S_{\theta},d)$ is complete and separable, let ${\cal F}_\theta$ be the Borel $\sigma$-algebra.

Let $V\colon \R^d \to \R$ be a non-negative function, that we call the potential. The Hamiltonian associated to $V$ restricted to the set $\La$ is
\begin{equation} \label{hamiltoniano}
H_{\theta, \La}(\s) = \sum\limits_{s \in \La} V\Big(X(\s(s))-X(s)\Big)\,.
\end{equation} 
For $\xi \in S_{\theta}$ and $\La \Subset \Z^d$, consider the set of permutations that are compatible with the boundary condition $\xi$ at volume $\La$, given by
\begin{equation}
\label{compatible}
S_{\theta, \La}^{\xi}:= \{\s \in S_{\theta} \colon \s^n(s) = \xi^n(s) \text{ for all } s\in \La^c, \, n\in \Z\},
\end{equation} 
where $\s^n$ means the $n$-th fold composition of $\s$ with itself. When $\La \Subset \Z^d$ the set $S_{\theta, \La}^{\xi}$ is also finite. We illustrate these definitions for a particular choice of $\Lambda$, $\theta$ and boundary condition $\xi$ in Figure~\ref{Figura}. 

\begin{center}
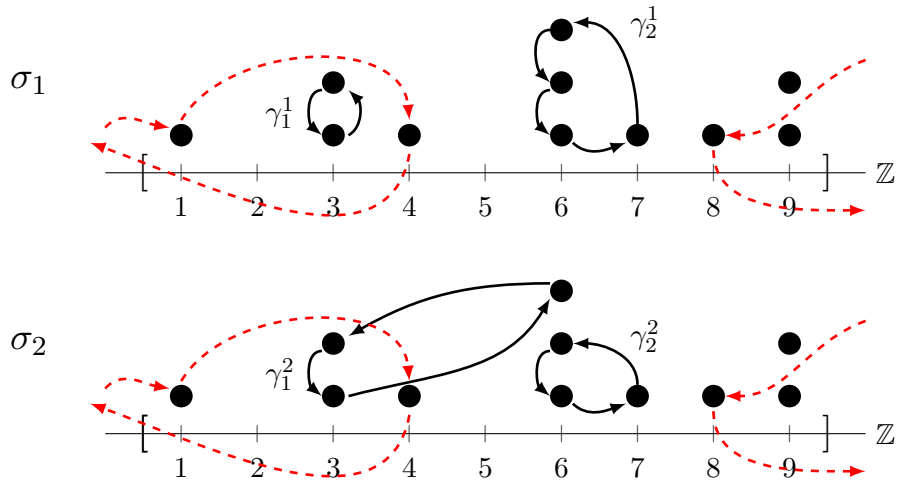
\begin{figure}[h]
\begin{center}
\begin{tikzpicture}
\draw (-1,1.2) node[scale=1.3]{$\sigma_1$};
\draw[-] (0,0)--(10,0) node[right]{$\Z$};
\node at (2.3,0.8) {$\gamma^1_1$};
\node at (7.1,2) {$\gamma^1_2$};
\draw (0.5,0) node[scale=1.3]{$[$};
\draw (9.5,0) node[scale=1.3]{$]$};
\foreach \x in {1,2,...,9}
{\draw (\x,0) node[scale=0.6]{$|$};
\draw (\x,-0.2) node[below]{$\x$};}
\fill[black] (1,0.5) circle(0.15);
\fill[black] (3,0.5) circle(0.15);
\fill[black] (3,1.2) circle(0.15);
\fill[black] (4,0.5) circle(0.15);
\fill[black] (6,0.5) circle(0.15);
\fill[black] (6,1.2) circle(0.15);
\fill[black] (6,1.9) circle(0.15);
\fill[black] (7,0.5) circle(0.15);
\fill[black] (8,0.5) circle(0.15);
\fill[black] (9,0.5) circle(0.15);
\fill[black] (9,1.2) circle(0.15);
\draw[->,>=latex, color=black, line width=1pt] (3.2,0.5) to[out=30,in=-45] (3.2,1.1);
\draw[->,>=latex, color=black, line width=1pt] (2.85,1.1) to[out=180,in=140] (2.85,0.5);
\draw[->,>=latex, color=black, line width=1pt] (6.15,0.4) to[out=-45,in=200] (6.85,0.4);
\draw[->,>=latex, color=black, line width=1pt] (7,0.65) to[out=90,in=0] (6.15,2);
\draw[->,>=latex, color=black, line width=1pt] (5.85,1.9) to[out=180,in=135] (5.85,1.2);
\draw[->,>=latex, color=black, line width=1pt] (5.85,1.1) to[out=180,in=140] (5.85,0.5);
\draw[->,>=latex, color=red, line width=1pt,dashed] (0,0.6) to[out=50,in=170] (0.85,0.6);
\draw[->,>=latex, color=red, line width=1pt,dashed] (1.0,0.7) to[out=60,in=95] (4.0,0.7);
\draw[->,>=latex, color=red, line width=1pt,dashed] (4.0,0.25) to[out=-100,in=-20] (-0.2,0.4);
\draw[->,>=latex, color=red, line width=1pt,dashed] (10,1.5) to[out=200,in=0] (8.15,0.5);
\draw[->,>=latex, color=red, line width=1pt,dashed] (8,0.3) to[out=270,in=180] (10,-0.5);
\end{tikzpicture}
\begin{tikzpicture}
\draw (-1,1.2) node[scale=1.3]{$\sigma_2$};
\node at (2.3,0.8) {$\gamma^2_1$};
\node at (7.1,1.3) {$\gamma^2_2$};
\draw (0.5,0) node[scale=1.3]{$[$};
\draw[-] (0,0)--(10,0) node[right]{$\Z$};
\draw (0.5,0) node[scale=1.3]{$[$};
\draw (9.5,0) node[scale=1.3]{$]$};
\foreach \x in {1,2,...,9}
{\draw (\x,0) node[scale=0.6]{$|$};
\draw (\x,-0.2) node[below]{$\x$};}
\fill[black] (1,0.5) circle(0.15);
\fill[black] (3,0.5) circle(0.15);
\fill[black] (3,1.2) circle(0.15);
\fill[black] (4,0.5) circle(0.15);
\fill[black] (6,0.5) circle(0.15);
\fill[black] (6,1.2) circle(0.15);
\fill[black] (6,1.9) circle(0.15);
\fill[black] (7,0.5) circle(0.15);
\fill[black] (8,0.5) circle(0.15);
\fill[black] (9,0.5) circle(0.15);
\fill[black] (9,1.2) circle(0.15);
\draw[->,>=latex, color=black, line width=1pt] (2.85,1.1) to[out=180,in=140] (2.85,0.5);
\draw[->,>=latex, color=black, line width=1pt] (3.2,0.5) to[out=15,in=-120] (5.85,1.8);
\draw[->,>=latex, color=black, line width=1pt] (5.85,2) to[out=180,in=30] (3.2,1.3);
\draw[->,>=latex, color=black, line width=1pt] (6.15,0.4) to[out=-45,in=200] (6.85,0.4);
\draw[->,>=latex, color=black, line width=1pt] (7,0.65) to[out=90,in=0] (6.15,1.2);
\draw[->,>=latex, color=black, line width=1pt] (5.85,1.1) to[out=180,in=140] (5.85,0.5);
\draw[->,>=latex, color=red, line width=1pt,dashed] (0,0.6) to[out=50,in=170] (0.85,0.6);
\draw[->,>=latex, color=red, line width=1pt,dashed] (1.0,0.7) to[out=60,in=95] (4.0,0.7);
\draw[->,>=latex, color=red, line width=1pt,dashed] (4.0,0.25) to[out=-100,in=-20] (-0.2,0.4);
\draw[->,>=latex, color=red, line width=1pt,dashed] (10,1.5) to[out=200,in=0] (8.15,0.5);
\draw[->,>=latex, color=red, line width=1pt,dashed] (8,0.3) to[out=270,in=180] (10,-0.5);
\end{tikzpicture}
\end{center}
\caption{Each figure above depicts a permutation in $S_{\theta,\La}^{\xi}$, $\Lambda = \{1,\dots,9\} \Subset \Z$. Cycles in the decomposition of the boundary condition $\xi$ are drawn using dashed lines, and cycles contained in $\Lambda$ that are compatible with $\xi$ are drawn in solid lines. 
}\label{Figura}
\end{figure} 
\end{center}
The specification at volume $\La$ associated to temperature $\alpha>0$ and boundary condition $\xi$ is given by 
\begin{equation} \label{especificaciones}
G_{\theta, \La}^{\xi}(\s)= \frac{e^{-\alpha H_{\theta,\La}(\s)}}{Z_{\theta, \La}^{\xi}} \one \{\s\in S_{\theta, \La}^{\xi}\}\,,
\end{equation} 
where $Z_{\theta,\La}^{\xi}$ is a normalizing constant that depends on $\alpha$, $\xi$, $\theta$ and $\La$. A distribution $\mu$ on $(S_{\theta}, \mathcal{F}_{\theta})$ is a Gibbs measure with respect to the Hamiltonian~\eqref{hamiltoniano} and compatible with specifications~\eqref{especificaciones} if for any $\La \Subset \Z^d$ and $A\in \mathcal{F}_{\theta}$ we have 
\begin{equation*} 
\mu(A) = \int G_{\theta, \La}^{\xi}(A) \,\dd\mu(\xi)\,.
\end{equation*}
A cycle $\g$ associated to $(s_1, \dots, s_{n}) \in {\Omega^n_\theta}$, $n\in \NN\cup \{\infty\}$, is a permutation $\g \in S_\theta$ such that $\g(s_i) = s_{i+1}$, 
$i=1,\dots,n-1$, $\g(s_n)=s_1$ if $n<\infty$, and $\g(s)=s$ otherwise. We write $\gamma=(s_1, \dots, s_{n})$ for short. Due to the cyclic structure of the cycle, the choice of starting point in this representation is arbitrary. A permutation $\s$ is a {\it finite cycle permutation} when its decomposition consists of finite cycles. Let 
\begin{equation*}
S_{\theta}^F:=\text{set of finite cycle permutations of points in }S_{\theta}. 
\end{equation*}
Given a cycle $\g$, we define its {\it support} as
\begin{equation*}
\{\g\}:=\{s\in \Omega_\theta,\, \g(s)\neq s\}.
\end{equation*} Clearly $\{\g\}=\{s_1,\dots,s_n\}$ when  $\g$ is associated to $(s_1,\dots,s_n)$. We will abuse notation and write $s\in \g$ instead of $s\in \{\g\}$.
We denote by
\begin{align}
&\Gamma_{\theta}=\big\{ \gamma \text{ cycle in }S_\theta, \#\{\g\}<\infty \big\} \label{Gamma}, \\
&\Gamma_{\theta,\La}=\big\{ \gamma \in \Gamma_\theta,\,\{\gamma\}\subseteq \Lambda \big\}. \nonumber
\end{align}  
Given a permutation $\s$ we say that $\g\in\s$ if $\g$ is one of the cycles in its decomposition. We define the {\it weight} of a finite cycle $\gamma$ as
\begin{equation} \label{weight}
w(\gamma):=e^{-\alpha \sum_{x\in \gamma} V(s-\gamma(s))}.
\end{equation}

\paragraph{The quadratic potential case in the discrete setting.} We first consider the case $V(x)=\|x\|^2$.

\begin{teo}\label{Teo-cuadratico}
Let $\rho\in(0,1/2)$ be the density of the point process in $\Z^d$. If $\alpha>\alpha_*$, where $\alpha_*$ is defined below in~\eqref{alfa_*}, then for almost every realization of $\{\theta(x)\}_{x\in\Z^d}$ there exists a unique Gibbs measure $\mu_{\theta}$ that concentrates on finite cycle permutations. 

Furthermore, $\mu_{\theta}$ can be obtained as a subsequential weak limit of specifications with identity boundary conditions.

\end{teo}

Let $r_0\in [0,1]$ be the unique solution of the equation \begin{equation}\label{r_0}
\frac{r}{(1-r)^2}-r=\frac{1}{2},\quad r_0\approx 0.35542.
\end{equation} 
Given $\rho \in(0,1/2),$ let 
\begin{equation} \label{C_rho} 
C_{\rho}:= \frac{\rho e^{-\rho+\frac{1}{2}}}{1-2\rho},
\end{equation}
and define $\alpha_*$ as
\begin{equation}
\label{alfa_*}
\alpha_*:=\frac{\pi}{\left[\big(\frac{r_0}{C_\rho}+1\big)^{\frac{1}{d}}-1\right]^2}\,.
\end{equation} 

The proof of Theorem~\ref{Teo-cuadratico} holds for any environment given by an i.i.d. family $\{\theta(x)\}_{x\in\Z^d}$ such that $\E(\theta(x)!\,2^{\theta(x)}) < \infty$, including the case discussed here where $\theta(x)\sim \cal{P}(\rho)$.

\paragraph{The general potential case in the discrete setting.}

A result analogous to theorem~\ref{Teo-cuadratico} holds for general potentials $V$ under appropriate growth conditions. In particular, it is enough that there exists $\alpha_0\geq 0$ such that for $\alpha > \alpha_0$  
\begin{equation} \label{funcion-varphi-general} 
\varphi_V(\alpha) = \sum\limits_{x\in \Z^d} e^{-\alpha V(x)} < \infty\,.
\end{equation} 
Betz~\cite{Betz-criterio-existencia} proves that this condition implies tightness of the specifications when the set of points is the integer lattice $\Z^d$.

Note that $V$ need not be strictly positive, and then there are jumps of positive length that do not contribute to the Hamiltonian. In this case an infinite cycle could have finite energy. To avoid this problem we need to restrict the density. Define 
\begin{equation} \label{L_V}
L_V := \sup\{\|x\| \colon V(x)=0,\, x\in\Z^d\},
\end{equation} 
and note that under \eqref{funcion-varphi-general}, $L_V<\infty $. Given $R>0$, let $E(R)$ be the event that there exists an infinite sequence of points $(s_i)_{i\in \mathbb{N}} \subset \Omega_{\theta}$ such that $\|X(s_i) - X(s_{i+1})\|\leq R$. Meester and Roy~\cite[Chapter~3]{Meester-Roy} prove that for low enough density there is no Boolean percolation in the continuous Poisson model. A similar argument proves that there exists $\rho_c(R)$ with $\P(E(R))=0$ if $\rho<\rho_c(R)$ and $\P(E(R))=1$ when $\rho>\rho_c(R)$.

A density $\rho$ is called good for the potential $V$ when $L_V<1$ or $\P(E(L_V))=0$. 

\begin{teo}\label{Teo-V-general}
Let $V$ a non-negative potential satisfying~\eqref{funcion-varphi-general} and $\rho\in(0,1/2)$ a good density for $V$. Consider $\alpha$ and $\rho$ such that  
\begin{equation*} 
C_{\rho} \,\varphi_V(\alpha)< r_0\,,
\end{equation*}  where $r_0$ and $C_{\rho}$ are defined in~\eqref{r_0} and~\eqref{C_rho} respectively.

Then for almost every realization of $\{\theta(x)\}_{x\in\Z^d}$ there exists a unique Gibbs measure $\mu_{\theta}$ supported on finite cycle permutations. 
\end{teo}
The family of Hamiltonians that satisfy the hypotheses of Theorem \ref{Teo-V-general} includes the one-body potentials discussed in~\cite{Betz-Ueltschi-PD-EJP}.

\paragraph{The quadratic potential case in the continuous setting.}

Let us now consider the case when the set of points $\Omega$ is a realization of a Poisson point process on $\R^d$ with density $\rho$. The definitions of specifications and Gibbs measures are analogous to those in discrete setting. 

\begin{teo}
Let $\rho$ and $\alpha$ be as in the statement of Theorem~\ref{Teo-V-general} for the potential 
\begin{equation}
\label{2ndpotential}
V(x):=\max\big\{\|x\|^2-2\sqrt{d}\,\|x\|,\, 0\big\}.
\end{equation}
Then for almost every realization $\Omega$ of the Poisson point process at density $\rho$, there exists a unique Gibbs measure at temperature $\alpha$ supported on finite cycle permutations associated to the quadratic Hamiltonian $H(\s)=\sum_{x\in \Omega} \|x-\s(x)\|^2.$
\end{teo}

The proof follows as a corollary of Theorem~\ref{Teo-V-general} applied to $V$ in~\eqref{2ndpotential} and the discrete Poisson point process obtained by collecting all points in 
$x+[0,1)^d$ at $x$, $x\in \Z^d$.

\section{Domination by a Poisson process} \label{section-domination}

In this section we prove that the finite volume specifications \eqref{especificaciones} can be dominated by a Poisson measure on finite cycle permutations. The approach follows ideas from~\cite{FFG,FFG-spa,AFGL}. In these articles, the Gibbs measure in infinite volume is realized as the stationary distribution of a suitable Markov process on the space of gas of finite cycles. Extending this method to infinite volume to get a similar result in our case would require that 
$\beta(\alpha) = \sup_{s\in \Omega_{\theta}} \beta(\alpha, s)<1$, where 
\begin{equation*}
\beta(\alpha,s):= \sum_{\g \in \Gamma_{\theta}} |\g|w(\g)\, \one\{s\in \g\}, \hspace{6mm} s\in \Omega_{\theta},
\end{equation*} 
with $w(\g)$ the weight in~\eqref{weight} and $\Gamma_{\theta}$ the set in~\eqref{Gamma}. It is easy to see that this condition fails to hold almost surely.  Indeed, for any $s\in \Omega_{\theta}$ with $X(s)=x$, there exist $(\theta(x)-1)!$ cycles passing through $s$ that have zero weight, and hence $\beta(\alpha,s)\geq (\theta(X(s))-1)!$. Taking the supremum over $s\in \Omega_{\theta}$ we conclude that $\beta(\alpha)=\infty$ almost surely in $\{\theta(x)\}_{x\in \Z^d}$, for any $\alpha>0$.

Let $\Lambda \subset \Z^d$ be a finite set. A finite cycle permutation $\s \in S_{\theta}^F$ can be represented as a configuration $\eta\in \{0,1\}^{\Gamma_{\theta}}$ with $\eta(\g)= \one\{\g\in \s\}$. We say that $\eta$ is the {\it gas of cycles representation of} $\sigma$ and write $\g\in\eta$ iff $\eta(\g)=1$. Notice that  if $\eta(\gamma)=1$, then $\eta(\g')=0$ for any cycle $\g'$ with $\{\g\}\cap \{\g'\}\neq \emptyset$. 

Given $\xi \in S_{\theta}^{F}$ and $\La \Subset \Z^d$ let
\begin{equation*}
B(\xi, \La)= \{\g \in \xi \colon \{ \g\} \cap \La \neq \emptyset, \{\g\} \cap \La^c \neq \emptyset\},
\end{equation*} 
the set of cycles from $\xi$ that intersect $\La$ and $\La^c$. The set of permutations $S_{\theta,\La}^{\xi}$ that are compatible with $\xi$ at volume $\La$ introduced in 
\eqref{compatible} can now be described as
\begin{align*}
S_{\theta,\La}^{\xi} = \Big\{\eta \in \{0,1\}^{\Gamma_{\theta}} \colon & \eta(\g)=1 \text{ for all } \g\in B(\xi, \La), \nonumber \\ 
& \eta(\g)=0 \text{ if } \g\in \Gamma_{\theta,\La} \text{ and there exists } \g'\in B(\xi, \La) \text{ with } \{\g\}\cap \{\g'\} \neq \emptyset, \nonumber \\ 
& \eta(\g)\eta(\g')=0 \text{ if } \{\g\}\cap\{\g'\}\neq \emptyset \text{ for all } \g, \g'\in \Gamma_{\theta,\La} \notag \\
& \eta(\g)=\xi(\g) \text{ if }\{\g\}\subset \La^c \Big \}, 
\end{align*} 
and the specification at finite volume $\La$ with boundary condition $\xi$ \eqref{especificaciones} becomes
\begin{equation} \label{especificacion-gas-ciclos}
G_{\theta, \La}^{\xi}(\eta) = \frac{1}{Z_{\theta, \La}^{\xi}} \prod\limits_{\g \in \Gamma_{\theta,\La}} w(\g)^{\eta(\g)}\, \one \{ \eta\in S_{\theta,\La}^{\xi}\} \,.
\end{equation} 
From this viewpoint the specification $G_{\theta, \La}^{\xi}$ is a distribution over the space of gases of cycles, which interact by exclusion: if a cycle $\g$ is in the gas, then any other cycle that uses a point visited by $\g$ cannot be in the gas. For the rest of the article we indistinctly denote configurations by $\s$ or its associated gas of cycles $\eta$.

\paragraph{The free process.}  Let $\mathcal{N}$ be a Poisson process on $\Gamma_{\theta} \times \mathbb{R} \times \mathbb{R}_+$ with intensity measure $w(\g) \times dt \times e^{-r} dr$. Given $\xi$ and $\Lambda$, we define the {\it free process $(\eta_t^{o,\xi,\La} \colon t\in \R)$ on $\mathbb{N}_0^{\Gamma_{\theta}}$ associated to $\xi,\, \Lambda$} as 
\begin{equation} \label{xi-BD-process}
\eta_t^{o,\xi,\La}(\g) = \one\{\g \in B(\xi,\La)\} + \sum_{(\g,t',r')\in \mathcal{N}} \one\{t'\leq t < t'+r'\}\,.
\end{equation} 
each marginal process $(\eta_t^{o,\xi,\La}(\g) \colon t\in \R)$ is a birth and death process of cycles of type $\g$, shifted by 1 when $\g\in B(\xi, \La)$ to account for the boundary condition. A new copy of cycle $\g$ appears at rate $w(\g)$ and is removed at rate $1$, independently of other copies of the same cycle, and of other cycles. The generator of this process is given by
\begin{multline} \label{generador-free-xi}
\mathcal{L}^{o,\xi,\La} f(\eta) = \sum_{\g \in \Gamma_{\theta}} w(\g) \left[ f(\eta + \delta_{\g}) - f(\eta)\right] + \sum_{\g\in B(\xi, \La)} \eta(\g)\one\{\eta(\g)\geq 2\} \left[ f(\eta - \delta_{\g}) - f(\eta)\right] \\ + \sum_{\g\notin B(\xi, \La)} \eta(\g) \left[ f(\eta - \delta_{\g}) - f(\eta)\right] \,,
\end{multline}
$f:\mathbb{N}_0^{\Gamma_{\theta}}\to \R$ a test function.

Consider the product measure $\nu_{\theta,\La}^{\xi}$ on $\mathbb{N}_0^{\Gamma_{\theta}}$ such that, independently for $\g \in \Gamma_\theta,$ the marginal distribution 
$\nu_{\theta,\La}^{\xi}(\g)$ satisfies
\begin{align*}
\nu_{\theta,\La}^{\xi}(\g)&\sim \text{Poisson}\big(w(\g)\big)\quad \text{if } \g \not\in B(\xi, \La)\notag\\
\nu_{\theta,\La}^{\xi}(\g)&\sim 1+\text{Poisson}\big(w(\g)\big)\quad \text{if } \g \in B(\xi, \La).
\end{align*}
When $\xi=\text{id}$, the set $B(\text{id},\Lambda)=\emptyset$ for any $\Lambda\Subset \Z^d$, and $\nu_{\theta,\La}^{\text{id}}$ does not depend on $\Lambda$; in this case we  write $\nu_{\theta}$ instead of $\nu_{\theta,\La}^{\text{id}}$.
Each marginal $\nu_{\theta,\Lambda}^{\xi}(\g)$ is reversible for the birth and death dynamics of cycles of type $\g$, hence the product measure $\nu_{\theta,\La}^{\xi}$ is reversible for $\mathcal{L}^{o,\xi,\La}$.

The family of processes $\big\{(\eta_t^{o,\xi,\Lambda} \colon t\in \R),\, \xi\in S_\theta^F\big\}$ can be simultaneously built using 
the same driving Poisson process $\mathcal{N}$. The coupled construction yields $\eta_t^{o,\xi,\Lambda} \geq \eta_t^{o,\text{id},\Lambda}$, that is, 
\[
\eta_t^{o,\xi,\Lambda} (\g)\geq \eta_t^{o,\text{id},\Lambda}(\g)\quad \text{ for all }\g\in \Gamma_{\theta},\, t\in \R.
\]
In fact, these processes only differ if $\g \in B(\xi, \La)$, and then $\eta_t^{o,\xi,\Lambda} (\g)= \eta_t^{o,\text{id},\Lambda}(\g)+1$.

Note that $\nu_{\theta,\La}^{\xi}$ assigns positive probability to $S_{\theta,\La}^{\xi}$, hence the conditional measure $\nu_{\theta,\La}^{\xi}(\,\cdot \,|S_{\theta, \La}^{\xi})$ is well-defined, and it is a simple computation to verify that $\nu_{\theta,\La}^{\xi}(\,\cdot \,|S_{\theta, \La}^{\xi})= G_{\theta,\La}^{\xi}(\cdot)$.

\paragraph{The loss network.}
Our goal now is to define a process that can be easily compared to the free process, and which has $G_{\theta, \La}^{\xi}$ as invariant measure. We will realize it as a thinning of the free process. As before, let $\xi \in S_\theta^F$ and $\La \Subset \Z^d$ be fixed.

We say that two cycles $\g$ and $\g'$ are compatible if $\{\g\}\cap \{\g'\}= \emptyset$. Otherwise, they are incompatible. A cycle $\g$ is compatible with the gas of cycles $\eta \in \{0,1\}^{\Gamma_{\theta}}$, which we denote by $\g\sim\eta$, when $\g$ is compatible with all cycles $\g'$ such that $\eta(\g')=1$. 

The {\it loss network associated to $\xi,\, \Lambda$}  is the Markov process in $S^\xi_{\theta,\La}$ with generator
 \begin{equation} \label{xi-generador-loss}
\mathcal{L}^{\xi,\La} f(\eta) = \sum_{\g \in \Gamma_{\theta, \La}} w(\g) \one\{\g \sim \eta\}\left[ f(\eta + \delta_{\g}) - f(\eta)\right] + \sum_{\g\in \Gamma_{\theta,\La}} \eta(\g) \left[ f(\eta - \delta_{\g}) - f(\eta)\right]\,,
\end{equation} 
$f\colon S_{\theta, \La}^{\xi} \mapsto \R$ a test function.

Informally, the loss network follows the dynamics of the free process but it is subject to an exclusion rule: a cycle $\g\in\Gamma_{\theta,\La}$ tries to be added at rate $w(\g)$ but the attempt is effective only when $\g$ is compatible with the cycles already present in the configuration at the time; each cycle is removed at rate 1 independently of others; and a copy of each cycle in $B(\xi, \La)$ is present at all times. The loss network is an irreducible Markov process in a finite state space with a unique invariant measure.

\begin{lema} Let $\La\Subset \Z^d$ and $\xi$ a finite cycle permutation. The measure $G_{\theta, \La}^{\xi}$ defined in \eqref{especificacion-gas-ciclos} is the unique invariant distribution for the generator $\mathcal{L}^{\xi,\La}$. \end{lema}

Since we only need to verify the detailed balance equations, we omit the proof.

Denote by $\eta_t^{\xi,\La}$ the loss network process related to $\xi$ and $\La$ at time $t$. We want to construct $\eta_t^{\xi,\La}$ using a convenient thinning of the free process $(\eta_t^{o,\xi,\La})$ to obtain $\eta_t^{\xi,\La} \leq \eta_t^{o,\xi,\La}$ for all $t$. The algorithm to delete cycles needs to know if the birth attempt of a cycle is allowed or not. So, we consider the clan of ancestors of $\zeta=(\g,t,r) \in \Gamma_{\theta, \La}\times \R \times \R^+$ as follows. The first generation of ancestors supported on $\La$ is defined by: \begin{equation*}
A_1^{\zeta, \La} = \{(\g',t',r')\in \mathcal{N} \colon \g'\in \Gamma_{\theta, \La},\, \g' \nsim \g, \, t'<t<t'+r'\}\,. 
\end{equation*} 
Inductively, if $A_{n-1}^{\zeta, \La}$ is determined, for the $n$-th generation we set: \begin{equation*} 
A_{n}^{\zeta, \La} = \bigcup\limits_{\upsilon \in A_{n-1}^{\zeta, \La}} A_{1}^{\upsilon, \La}\,.
\end{equation*} The clan of ancestors of the mark $\zeta$ supported in $\La$ is defined by 
\begin{equation*} 
A^{\zeta,\La} = \bigcup\limits_{n \geq 1} A_n^{\zeta,\La}. 
\end{equation*}
Suppose that $A^{\zeta,\La}$ is finite for all $\zeta \in \Gamma_{\theta, \La}\times \R \times \R^+$ and for almost all realizations of $\mathcal{N}$. To describe the thinning of $\eta^{o,\xi,\Lambda}$ we define in each step if a cycle is kept or deleted using its clan of ancestors.

Let $\mathcal{D}_0^{\xi,\La}= \{(\g,t,r) \in \mathcal{N} \colon \g\nsim \g' \text{ for some } \g'\in B(\xi,\La)\}$ and for $n\geq 1$ set \begin{equation}\label{kept-deleted-paso-n}
\mathcal{K}_{n}^{\xi,\La} = \{\zeta \in \mathcal{N} \colon A_1^{\zeta, \La} \setminus \mathcal{D}_{n-1}^{\xi, \La} = \emptyset \}, \hspace{1cm} \mathcal{D}_{n}^{\xi,\La} = \{\zeta \in \mathcal{N} \colon A_1^{\zeta, \La} \cap \mathcal{K}_{n}^{\xi,\La} \neq \emptyset \} \,.
\end{equation} 

Let $\mathcal{K}^{\xi,\La}= \cup_{n\geq 1} \mathcal{K}_n^{\xi,\La}$ be the set of kept cycles and $\mathcal{D}^{\xi, \La}= \cup_{n \geq 1} \mathcal{D}_n^{\xi,\La}$ be the set of deleted cycles. Note that in the initial step any cycle that is incompatible with a cycle from $B(\xi,\Lambda)$ is deleted. Under the assumption that all the clans of ancestors supported in $\La$ are finite, every mark $\zeta \in \Gamma_{\theta, \La}\times \R \times \R^+$ is kept or deleted.

Now, using kept cycles we give a graphical construction for the loss network related to $\xi$ at volume $\La$ by the formula \begin{equation} \label{loss-network}
\eta_t^{\xi,\La}(\g) = \sum_{(\g,t',r')\in \mathcal{N}} \one\{t'\leq t < t'+r'\} \, \one\{(\g,t',r') \in \mathcal{K}^{\xi,\La}\} \, \one\{\g \in \Gamma_{\theta, \La}\}\,.
\end{equation} 
To show that~\eqref{loss-network} is well-defined we need to check that the clan of ancestors of any mark $(\g,t',r')$ is finite almost surely. The next Lemma proves it when $\La$ is finite but unfortunately the argument does not work when $\La$ is infinite as we seen at the beginning of this section.

\begin{lema} \label{xi-especificacion-invariante}
If $\La\Subset \Z^d$, the process $(\eta_t^{\xi,\La} \colon t\in \R)$ is well-defined. It is a Markov process with generator given by~\eqref{xi-generador-loss}. The construction~\eqref{loss-network} is stationary, so, $\eta_{t}^{\xi,\La}$ is distributed according to $G_{\theta, \La}^{\xi}$ for all $t$.
\end{lema}
\begin{proof} Since $\Gamma_{\theta, \La}$ is a finite set, for almost every realization of the process $\mathcal{N}$ there exists a sequence of times $\{t_j \colon j\in \Z\}$ with $t_{j} {\to} \pm\infty$ as $j \to \pm\infty$ such that $\eta_{t_j}^{o,\xi,\Lambda}(\g)=0$ for all $\g \in \Gamma_{\theta, \La}$. Therefore, $A^{\zeta, \La}$ must be finite almost surely for all $\zeta \in \Gamma_{\theta, \La}\times \R \times \R^+$.

If the process $(\eta_t^{o,\xi,\La} \colon t\in \R)$ is restricted to cycles in $\Gamma_{\theta, \La}$, marks can be sorted by their birth time (the second coordinate of the mark), and so, the algorithm described in~\eqref{kept-deleted-paso-n} works.

The construction~\eqref{loss-network} is stationary and Lemma~\ref{xi-especificacion-invariante} implies that $\eta_t^{\xi, \La}$ has distribution $G_{\theta,\Lambda}^{\xi}$ for all $t$.
\end{proof}

\begin{lema} \label{dominacion-de-especificaciones} Let $\xi$ be a finite cycle permutation. For almost every realization of the environment $\theta$, we have that $G_{\La, \theta}^{\xi}$ is stochastically dominated by $\nu_{\theta,\La}^{\xi}$ for all $\La\Subset \Z^d$.
\end{lema}
\begin{proof} Since $\nu_{\theta,\La}^{\xi}$ and $G_{\La, \theta}^{\xi}$ are invariant measures for the free process and the loss network process respectively, it is enough to give a coupling such that $\eta_t^{\xi,\La} \leq \eta_t^{o,\xi,\La}$ for all $t$. The coupling is to use the same $\mathcal{N}$ for both graphical representations defined in~\eqref{xi-BD-process} and~\eqref{loss-network}. By these constructions it follows that \[\eta_t^{\xi,\La}(\g) \leq \eta_t^{o,\xi,\La}(\g) \text{ for all } t,\] and as both process are stationary $\nu_{\theta,\La}^{\xi}$ dominates $G_{\La, \theta}^{\xi}$.
\end{proof}


\section{Existence of Gibbs measures}\label{section-existence}

In this section we prove that for $\rho\in(0,1/2)$ the family of specifications with identity boundary condition is tight in the large temperature regime for almost every realization of $\theta$. The tightness also holds considering specifications with a boundary condition given by a finite cycle permutation. In the next section we prove that for the large temperature regime there exists a unique Gibbs measure that concentrates over finite cycle permutations, so, weak limits of specifications with a general finite cycle boundary condition are the same as the identity case. For this reason, we focus on the identity boundary condition case.

All proofs are done for the quadratic potential case but work in the general case with slight modifications. In such cases we explain differences in the next subsection.

We have cycles that are different but they use the same sites and have an identical value for the Hamiltonian. See for instance $\g_1^2$ and $\g_2^2$ in Figure~\eqref{Figura}. We want to study the number of cycles that project to a fixed and ordered set.

The ordered support $[\gamma]$ of a cycle $\gamma=(s_1,\,s_2,\dots, s_n)$ is the vector in $(\Z^d)^m,\, m\le n$, given by
\begin{align} \label{ord-supp}
&[\gamma]=(x_1, x_2, \dots, x_m) \quad \text{with} \quad x_i=X(s_{\pi(i)})
\end{align} where $\pi(1)=1$, and inductively, $\pi(i)=\inf\{k>\pi(i-1),\, X(s_k)\neq X(s_{\pi(i-1)})\},\,\, i>1$. 
In other words, $[\g]$ is the projection of $\g$ to $\Z^d$ erasing consecutive repetitions of sites. Note that both in the representation of $\gamma$ as a vector and in the definition of its ordered support $[\gamma],$ due to the cyclic property of $\gamma$, the choice of initial point is arbitrary. Starting from any other point $s \in \{\gamma\}$ for the former, or of its spatial coordinate $X(s)$ for the latter, lead to alternative representations of the cycle and its ordered support. 

The following refers to Figure 1. The supports of the cycles in $\sigma_1$ are $\{\gamma_1^1\}=\{(3,1);(3,2)\}$ and 
$\{\gamma_2^1\}=\{(6,1);(6,2);(6,3);(7,1)\}$, and the ordered supports are $[\gamma_1^1]=(3)$ and $[\gamma_2^1]=(6,7)$. Also, $\{\gamma_2^2\}=\{(6,1);(6,2);(7,1)\}\neq \{\gamma_2^1\}$, but they share the ordered support, $[\gamma_2^1]=[\gamma_2^2]$.

Let $\bar{y} \in (\Z^d)^m$ be a vector such that $\bar{y}_i\neq \bar{y}_{i+1}$. It will represent an ordered support, so, it could have repetitions in different (but non-consecutive) coordinates. Write $N_{\theta}(\bar{y})$ for the number of cycles $\g$ such that $[\g]= \bar{y}$. In the appendix, using basic facts from combinatorics, we compute an upper bound for $N_{\theta}(\bar{y})$ and we call it $M_{\theta}(\bar{y})$. The explicit definition of $M_{\theta}(\bar{y})$ is in~\eqref{cotaM(y)}.

In the following we use that the expectation of $M_{\theta}(\bar{y})$ under $\P$ is bounded for $\rho\in(0,1/2)$ by
\begin{align}
\E[M_{\theta}(\bar{y})] \leq \left(\frac{\rho e^{-\rho+\frac{1}{2}}}{1-2\rho}\right)^{|\bar{y}|} \label{cota-expectation-M(y)}\,,
\end{align} where $|\bar{y}|$ is the number of coordinates of $\bar{y}$.

Note that the weight of a cycle is a function of its ordered support. So, for a cycle $\g$ with $[\g]=\bar{y}$ we have that $w(\g)=w(\bar{y}):= \exp\{-\alpha \sum_{i=1}^m \|y_{i+1} - y_i\|^2 \}$ assuming $y_{m+1}=y_1$. 

In certain situations it will be useful to sum the weights of all finite cycles that contain a site $x$. Instead, using the bound $M_{\theta}(\bar{y})$, we will sum over the ordered supports, whose sums are easier to calculate. In fact, the sum of the weights of ordered supports that contain site $x$ and have length $m$ is, 
\begin{equation} \label{suma-pesos-largo-m}
\sum\limits_{\substack{\bar{y} \colon x\in \bar{y}\\|\bar{y}|=m \\ \bar{y}_i\neq \bar{y}_{i+1}}} w(\bar{y}) = \sum\limits_{\substack{\bar{y} \colon x\in \bar{y}\\|\bar{y}|=m \\ \bar{y}_i\neq \bar{y}_{i+1}}} \prod\limits_{i=1}^{m} e^{-\alpha \|\bar{y}_i-\bar{y}_{i+1}\|^2} = \sum\limits_{\substack{t_1\dots t_m\\t_i \neq 0_d}} \prod\limits_{i=1}^{m} e^{-\alpha \|t_i\|^2} = \varphi(\alpha)^m,
\end{equation} 
where $\varphi(\alpha)= \sum_{t \in \Z^d,\, t \neq 0_d} e^{-\alpha \|t\|^2}$. Observe that $\varphi$ is the function defined in~\eqref{funcion-varphi-general} for the quadratic potential. It is easy to compute that $\varphi(\alpha) < (1+ \sqrt{\frac{\pi}{\alpha}})^d -1$. So, $\varphi$ is a decreasing function of $\alpha$ that tends to 0 when $\alpha \to +\infty$.

Now we start with a series of lemmas to prove the tightness of $\{G_{\theta, \La}^{\text{id}}\}_{\La \Subset \Z^d}$.

For $f: \Z^d \mapsto \mathbb{N}$ we define the set $\widehat{K}_f = \bigcap_{x\in \Z^d} \widehat{K}_f(x),$ where \begin{equation*} 
\widehat{K}_{f}(x)= \{\eta\in \mathbb{N}_0^{\Gamma_{\theta}} \colon \forall \,\g\in\eta \text{ such that } x\in \g \text{ we have } H(\g) \leq f(x)\} \,.
\end{equation*} Denote by $\widehat{K}_{f}^c(x)$ the complement of $\widehat{K}_{f}(x)$.

\begin{lema} \label{lema-cota-annealed} Let $\rho\in (0,1/2)$ and $\alpha>0$ such that 
\begin{equation} \label{condicion-existencia-rho-alpha} C_{\rho}\,\varphi(\alpha/2)<1\,, 
\end{equation} where $\varphi$ is the function defined in~\eqref{funcion-varphi-general} and $C_{\rho}=\frac{\rho e^{-\rho+\frac{1}{2}}}{1-2\rho}$. Then, \begin{equation} \label{cota-annealed}
\E \left[\nu_{\theta}(\widehat{K}_{f}^c(x)) \right] \leq C(\rho,\alpha) e^{-\frac{\alpha}{2} f(x)}\,.
\end{equation}
\end{lema}

In the quadratic potential case we know that $\varphi(\alpha/2) \to 0$ when $\alpha \to +\infty$, so, for any $\rho \in (0,1/2)$ we can choose $\alpha$ large enough such that $C_{\rho}\,\varphi(\alpha/2)<1$.

For a general potential $V$ the condition~\eqref{condicion-existencia-rho-alpha} becomes to $C_{\rho}\varphi_V(\alpha/2)<1$ where $\varphi_V$ is the analogous function of $\varphi$ defined in~\eqref{funcion-varphi-general}. Observe that now, $\varphi_V(\alpha)$ does not tend to 0 necessarily.

\begin{proof}[Proof of Lemma~\eqref{lema-cota-annealed}] 

By the cycle gas representation and using the marginal distributions of $\nu_{\theta}$ we have \begin{equation*}
\E \left[\nu_{\theta}(\widehat{K}_{f}^c(x)) \right] \leq \E \big[ \sum\limits_{\substack{\g \colon \g \ni x \\ H(\g)> f(x)}} \nu_{\theta}( \one\{\g \in \eta \} ) \big] = \E\big[ \sum\limits_{\substack{\g \colon \g \ni x \\ H(\g)> f(x)}} (1-e^{-w(\g)})\big] \,.
\end{equation*}

We want to sum over ordered supports instead of cycles. Write the weight of each cycle as a function of its order support, and recall that for each ordered support $\bar{y}$, the number of cycles that have ordered support $\bar{y}$ is bounded above by $M_{\theta}(\bar{y})$, where $M_{\theta}(\bar{y})$ was defined in~\eqref{cotaM(y)}. Then, using the linearity of $\E$, the bound~\eqref{cota-expectation-M(y)} and $1-e^{-t} \leq t$, we obtain 
\begin{equation*}
\E \big[\nu_{\theta}(\widehat{K}_{f}^c(x)) \big] \leq \E\big[ \sum_{m\geq 2} \sum\limits_{\substack{\bar{y} \colon \bar{y} \ni x \\ \bar{y}_i \neq \bar{y}_{i+1} \\ |\bar{y}|=m \\ H(\bar{y})> f(x)}} M_{\theta}(\bar{y}) (1-e^{-w(\bar{y})}) \big] \leq \sum_{m\geq 2} \sum\limits_{\substack{\bar{y} \colon \bar{y} \ni x \\ \bar{y}_i \neq \bar{y}_{i+1} \\ |\bar{y}|=m \\ H(\bar{y})> f(x)}} C_{\rho}^m w(\bar{y})\,.
\end{equation*} 
Now, using that $H(\bar{y})> f(x)$ and the definition of $\varphi$ (see~\eqref{suma-pesos-largo-m}) we have
\begin{equation*}
\E \big[\nu_{\theta}(\widehat{K}_{f}^c(x)) \big] \leq e^{-\frac{\alpha}{2} f(x)} \sum_{m\geq 2} C_{\rho}^m \sum\limits_{\substack{\bar{y} \colon \bar{y} \ni x \\ \bar{y}_i \neq \bar{y}_{i+1} \\ |\bar{y}|=m}} e^{-\frac{\alpha}{2} H(\bar{y})} = e^{-\frac{\alpha}{2} f(x)} \sum_{m\geq 2} C_{\rho}^m \, \varphi\left(\alpha/2\right)^{m}.
\end{equation*} Finally, lemma follows from the fact $C_{\rho} \varphi\left(\alpha/2\right)<1$.
\end{proof}

\begin{lema} \label{lema-cota-quenched} Let $\rho$ and $\alpha$ be in the same conditions as in Lemma~\eqref{lema-cota-annealed}. Given $\epsilon >0$, for almost every realization of $\theta$ there exists a function $f$ that depends on $\theta$, such that $\nu_{\theta}(\widehat{K}_{f}^c)<\epsilon$.
\end{lema}
\begin{proof} For each $n$, using~\eqref{cota-annealed} we pick $f_n(x)$ large enough such that: $\E \big[\nu_{\theta}(\widehat{K}_{f_n}^c(x)) \big] \leq 1/(n^2 2^{\|x\|})$. 
If we define $f_n: \Z^d \mapsto \mathbb{N}$ in the obvious way, we have: \begin{equation*}
\E \big[\nu_{\theta}(\widehat{K}_{f_n}^c) \big] \leq \E \big[\sum_{x\in \Z^d}\nu_{\theta}(\widehat{K}_{f_n}^c(x)) \big] \leq \sum_{x\in \Z^d} \frac{1}{n^2} \frac{1}{2^{\|x\|}} = \frac{C}{n^2}\,.
\end{equation*} 
So, we have a sequence of functions $\{f_n\}_{n \geq 1}$ such that $\E \big[ \sum_{n\geq 1} \nu_{\theta}(\widehat{K}_{f_n}^c) \big] <\infty \,.$
Therefore $\sum_{n\geq 1} \nu_{\theta}(\widehat{K}_{f_n}^c) <\infty$ for almost every realization of $\theta$. So, for almost every realization of $\theta$ there exists $n_0(\theta, \epsilon)$ such that $\nu_{\theta}(\widehat{K}_{f_n}^c) <\epsilon$ for all $n\geq n_0$.
\end{proof}

The following lemma works in the quadratic potential case. Later we discuss the proof in the case of a general potential.

\begin{lema} \label{lema-Kf-compacto}
  Consider a function $f: \Z^d \mapsto \mathbb{N}$ and the set $K_f:= \widehat{K}_f \cap S_{\theta}^F$, where $S_{\theta}^F$ is the finite cycle permutation space. Then, $K_f$ is a non-empty compact set.
\end{lema}

\begin{proof} The gas of cycles representation for the identity is the null configuration, thus $\text{id} \in \widehat{K}_f(x)$ for all $x$ and any choice of $f(x)$.

Recall that a sequence of permutations $\{\s_n\}$ converges to $\s$ if and only if $\s_n(s)\to \s(s)$ for all $s\in \Omega_{\theta}$. Let $\{\s_n\}_{n\geq 1}$ be a sequence on $K_f$ and fix $s\in \Omega_{\theta}$. Definition of $K_f$ implies that \[ \|\s_n(s)-s\|^2 \leq \|X(\s_n(s))- X(s)\|^2 + R(s) \leq f(X(s)) + R(s),\] where $R(s)= \max\{\theta(x) \colon \|x-X(s)\|^2 \leq f(X(s))\}$. So, $\{\s_n(s)\}_{n\geq 1}$ is bounded and it has a convergent subsequence. Since $\Omega_{\theta}$ is countable and $S_{\theta}$ is a complete metric space, a Cantor's diagonal argument follows to prove that $\{\s_n\}$ has a subsequential limit $\s$. We write also $\{\s_n\}$ for the convergent subsequence.

It only remains to prove that $\s$ is a finite cycle permutation. On the contrary, suppose that the point $s$ is contained on an infinite cycle, i.e., $\s^j(s)\neq s$ for all $j\in \Z$. Choose a subsequence $\{{j_k}\}$ such that $X(\s^{j_k}(s))\neq X(\s^{j_{k-1}}(s))$ for all $k$. Since $\s_n \to \s$ we can pick $n$ large enough such that $\s_n^{j_k}(s) = \s^{j_k}(s)$ for all $k=1, \dots, k_0$ being $k_0=f(X(s))+1$. Then, calling $\g'$ for the cycle that contains $s$ in the permutation $\s_n$: \[H(\g')\geq \sum_{k=1}^{k_0} \|X(\s_n^{j_k}(s))-X(\s_n^{j_k-1}(s))\|^2 = \sum_{k=1}^{k_0} \|X(\s^{j_k}(s))-X(\s^{j_k-1}(s))\|^2 \geq f(X(s))+1\,,\] which contradicts that $\s_n\in K_f$. Hence, $\s$ is a finite cycle permutation.
\end{proof}

Recall that $S_{\theta}^F$ has a natural partial order, i.e., $\eta \leq \eta'$ if $\eta(\g)\leq \eta'(\g)$ for all $\g \in \Gamma_{\theta}$. So, an event $A\subset S_{\theta}^F$ is increasing when $\one_A$ is an increasing function with respect to the partial order.

\begin{lema} \label{lema-Kf-creciente} Set $K_f^c = \widehat{K}_f^c \cap S_{\theta}^F$. Then $K_f^c$ is an increasing event.
\end{lema}

\begin{proof} Fix $\eta \in \widehat{K}_f^c(x)$ and let be $\eta'$ such that $\eta \leq \eta$. By definition of $\widehat{K}_f^c(x)$ there exists $\g\in \eta$ such that $x\in\g$ and $H(\g)>f(x)$. The fact $\g\in \eta$ implies that $\g\in \eta'$, and so, using the same cycle $\g$ ones proves that $\eta'\in \widehat{K}_f^c(x)$. 
\end{proof}

The next result is a general fact from Gibbs measures theory, for this reason we omit its proof. 

\begin{lema} \label{lim-debil-es-gibbs} 
Let $\{\La_n\}_{n\geq 1} \Subset \Z^d$ an increasing sequence such that $\La_n \uparrow \Z^d$ and $\{G_{\theta, \La_n}^{\text{id}}\}_{n\geq 1}$ converges weakly to a probability measure $\mu$. Then, $\mu$ is a Gibbs measure.
\end{lema}

\begin{lema} Let $\rho$ and $\alpha$ satisfying conditions of Lemma~\eqref{lema-cota-annealed}. Then for almost every realization of $\theta$ there exists a Gibbs measure $\mu_{\theta}$ related to temperature $\alpha$ and specifications defined in~\eqref{especificaciones}.
\end{lema}
\begin{proof} First we prove that the family of specifications $\{G_{\theta, \La}^{\text{id}}\}_{\La\Subset \Z^d}$ is tight. By Lemma~\eqref{dominacion-de-especificaciones} we know that $\nu_{\theta}$ stochastically dominates $G_{\theta,\La}^{\text{id}}$ for all $\La\Subset \Z^d$ and, so, as $K_f^c$ is an increasing event we obtain: \[\sup_{\La \Subset \Z^d} G_{\theta, \La}^{\text{id}}(K_f^c) \leq \nu_{\theta}(K_f^c)\,, \hspace{0.8cm} \theta \text{ a.s..} \] 

Given $\epsilon>0$, by Lemma~\eqref{lema-cota-quenched} for almost every realization of $\theta$ there exists a function $f$, such that $\nu_{\theta}(K_f^c)<\epsilon$. Combining this with the previous inequality and the compactness of $K_f$, tightness follows. Now, we have a subsequential limit $\mu_{\theta}$ and Lemma~\eqref{lim-debil-es-gibbs} shows that $\mu_{\theta}$ is a Gibbs measure.
\end{proof}

\subsection{Remarks for the general potential case}

Previous results hold also when the Hamiltonian is given by a potential $V$ satisfying~\eqref{funcion-varphi-general} instead the quadratic potential. The only relevant difference is in Lemma~\eqref{lema-Kf-compacto} to prove that $K_f$ is compact, concretely in the proof that limit permutation $\s$ has only finite cycles in its decomposition (other steps can be adapted). Other proofs work line by line, using $\varphi_V$ with $\alpha \geq \alpha_0$ in the statements instead of $\varphi$. 

To show that $\s$ is a finite cycle permutation we need to ensure that any infinite cycle has infinite energy associated to the Hamiltonian. In the case when $V$ is strictly positive on $[1,+\infty)$ the previous proof works, since each jump among points located in different sites contributes non-zero to $H_V$.

Otherwise, if $L_V\geq 1$, where $L_V= \sup\{\|x\| \colon V(\|x\|)=0\}$, it can exist infinite cycles with finite energy. To skip this problem we take the density $\rho$ small enough to ensure that the event that there exists an infinite sequence of points $\{s_i\}_{i\in \mathbb{N}} \subset \Omega_{\theta}$ such that $\|X(s_i) - X(s_{i+1})\|\geq L_V$, has zero probability with respect to the environment. This is exactly our definition that $\rho$ is good density for the potential $V$ in~\eqref{L_V} and the condition that appears in the Theorem~\eqref{Teo-V-general}.


\section{Uniqueness of Gibbs measures}\label{section-uniqueness}

Proved the existence of a Gibbs measure we want to prove that it is unique. Specifically, we want to show that if $\mu$ and $\mu'$ are Gibbs measures supported on the finite cycle permutations then $\mu=\mu'$. For that we consider the product measure $\mu\otimes \mu'$, which is a Gibbs measure with respect to the product specifications $G_{\theta, \La_1}^{\xi_1}\otimes G_{\theta, \La_2}^{\xi_2}$ with $\xi_1$, $\xi_2$ and $\La_1,\La_2\Subset \Z^d$. 

As in the previous section we focus in the quadratic potential case. For a potential $V$ the difference is to use $\varphi_V$ instead of $\varphi$ in each statement.

For the rest of this section $\mu$ and $\mu'$ will be Gibbs measures supported on finite cycle permutations.

\begin{df} We say that $\Delta \subset \Z^d$ separates $\eta \in \mathbb{N}_0^{\Gamma_{\theta}}$ when for all $\g\in \eta$ we have $\{\g\} \subset \Delta$ or $\{\g\} \subset \Delta^c$. For a pair $(\eta,\eta')$ we say that the pair is separated by $\Delta \subset \Z^d$ when both coordinates are separated by $\Delta$. 
\end{df}

The separating set property is closed by unions. Indeed, let $\Delta_1$ and $\Delta_2$ such that both are separating sets for $(\eta,\eta')$. If $\g\in \eta$ is such that $\{\g\}\subset \Delta_1, \Delta_2$, we have $\{\g\}\subset \Delta_1 \cup \Delta_2$. In other case, as $\Delta_1$ and $\Delta_1$ are separating sets, we have $\{\g\}\subset \Delta_1^c \cap \Delta_2^c$. So, $\Delta_1 \cup \Delta_2$ is a separating set for $\eta$ and the same holds for $\eta'$.

Denote by $\Lambda_l$ the box $[-l,l]^d\cap \Z^d$. Let $A_n$ be the event that exists $\Delta\Subset \Z^d$ such that $\Delta$ separates $(\eta,\eta')$ and $\Delta \supset \Lambda_n$. Note that $A_{n+1}\subset A_{n}$ and define $A=\cap_{n\geq 1} A_n$. Observe also that $A$ and $A_n$ are decreasing events with respect to the partial order of $\mathbb{N}_0^{\Gamma_{\theta}}$. 

Our goal is to prove that for sufficiently large $\alpha$ the event $A$ has full measure with respect to the product measure $\mu\otimes \mu'$. Then we use the existence of an arbitrary large separating set to prove that $\mu$ and $\mu'$ are equal to the weak limit of specifications with identity boundary condition.

We say that cycles $\g, \g' \in \Gamma_{\theta}$ are neighbors, and we write $\g\Join\g'$, if exists $s$, $s'\in \Omega_{\theta}$ such that $s\in \g$, $s'\in \g'$ and $X(s)=X(s')$. In terms of ordered supports $\g$ and $\g'$ are neighbors when $\{[\g]\}\cap\{[\g]'\}\neq \emptyset$. A path of cycles of length $n$ is a sequence of $n$ different cycles $\g_1, \dots,\g_n$ such that $\g_i \Join \g_{i+1}$ for $i=1,\dots,n-1$. The idea is to consider a random subgraph of $(\Gamma_{\theta}, \Join)$ and ask about percolation on it, i.e., the existence of an infinite path of cycles with positive probability.

Fix $(\eta,\eta')\in \NN^{\Gamma_{\theta}} \times \NN^{\Gamma_{\theta}}$. We declare that $\g$ is open when $\eta(\g)+ \eta'(\g)\geq 1$. A path is open when it is composed by open cycles. A cycle is a trivial cycle if only uses points located at the same site. We are interested in open paths that use non-trivial cycles.

Fix $x_0\in \Z^d$ with $\theta(x_0)\neq 0$. Let $D(n)$ be the event that there exists an open path of length $n$ formed only by non-trivial cycles and for which the first cycle contains $x_0$. Of course $D(n+1)\subset D(n)$ for all $n$. There are paths of length $n$ that are not included in $D(n)$ since we only allow non-trivial cycles. However, an infinite path of cycles exists if and only if there exists an infinite path using only non-trivial cycles. 

To enunciate the next lemma recall that $r_0$ is the unique solution in $[0,1]$ of the equation $\frac{r}{(1-r)^2}-r=\frac{1}{2}$. Note that $r<r_0$ implies $\sum_{m\ge 2} m r^m = \frac{r}{(1-r)^2}-r<\frac{1}{2}$.

\begin{lema} \label{no-percolacion-nu-times-nu} Suppose that $\rho\in(0,1/2)$ and $\alpha>0$ satisfy $C_{\rho} \varphi(\alpha)<r_0$, where $C_{\rho}$ is the constant that appears in Lemma~\eqref{lema-cota-annealed} and $\varphi$ is the function defined in~\eqref{funcion-varphi-general}. Let $\xi$, $\xi'$ be finite cycle permutations and $\Lambda \Subset \Z^d$.

Consider the pair $(\eta$,$\eta')$ independently sampled according to $\nu_{\theta,\Lambda}^{\xi} \otimes \nu_{\theta,\Lambda}^{\xi'}$ and the graph structure induced by it on $\Gamma_{\theta}$. 

Then for almost every realization of $\theta$, the event that there exists an infinite open path of cycles has zero probability with respect to $\nu_{\theta,\Lambda}^{\xi} \otimes \nu_{\theta,\Lambda}^{\xi'}$.

\end{lema}
\begin{proof} It is sufficient to show that $\lim_{n\to +\infty} \nu_{\theta}\otimes \nu_{\theta}(D(n))=0$. Indeed, there is a natural coupling between $\nu_{\theta}\otimes \nu_{\theta}$ and $\nu_{\theta,\Lambda}^{\xi} \otimes \nu_{\theta,\Lambda}^{\xi'}$ such that the former is stochastically dominated by the last one but both realizations have only a finite number of differences (see remark below~\eqref{generador-free-xi}). So, the graph induced by the realization of $\nu_{\theta,\Lambda}^{\xi} \otimes \nu_{\theta,\Lambda}^{\xi'}$ has more open cycles but only a finite number of them. So, the existence of an infinite open path under one measure is equivalent to the same event for the other.

Now, we compute the annealed expectation of $D(n)$ with respect to $\nu_{\theta} \otimes \nu_{\theta}$. Using the marginals distribution and independence, we obtain \begin{equation*}
\E[\nu_{\theta}\otimes \nu_{\theta}(D(n))] = \E[\sum_{\substack{ \g_1, \dots,\g_n}} \prod_{i=1}^n(1-e^{-2w(\g_i)})] \leq \E [\sum_{\substack{ \g_1, \dots,\g_n}} \prod_{i=1}^n 2w(\g_i)]\,, 
\end{equation*} where the sum denoted by $\sum_{\g_1, \dots,\g_n}$ is over all paths of length $n$ formed by non-trivial cycles and for which $x_0\in \g_1$. Any non-trivial cycle has an ordered support, so, we will sum over sequences of ordered supports instead cycles. Concretely, we sum over all sequences of $n$ ordered supports $\bar{y}_1, \dots, \bar{y}_n$ such that $x_0\in\{\bar{y}_1\}$ and $\{\bar{y}_i\}\cap \{\bar{y}_{i+1}\} \neq \emptyset$ for all $i$. This sum is denoted by $\sum_{\substack{\bar{y}_1, \dots, \bar{y}_n}}$. Recall that $M_{\theta}(\bar{y}_1 \dots \bar{y}_n)$ is an upper bound for the number of cycles $\g_1,\dots,\g_n$ that have ordered supports $\bar{y}_1,\dots,\bar{y}_n$ respectively (see Remark~\eqref{cota-indep-orden}). Combining these with the bound obtained in~\eqref{cota-expectation-M(y)} we have\begin{equation*}
\E[\nu_{\theta}\otimes \nu_{\theta}(D(n))] \leq \E[\sum_{\substack{\bar{y}_1, \dots, \bar{y}_n}} M_{\theta}(\bar{y}_1 \dots \bar{y}_n) \prod_{i=1}^n 2w(\bar{y}_i)] = \sum_{\substack{\bar{y}_1, \dots, \bar{y}_n}} \prod_{i=1}^n 2C_{\rho}^{|\bar{y}_i|}w(\bar{y}_i)\,.
\end{equation*} 
To estimate the last sum we need to consider all possibilities for which $\bar{y}_n$ shares a site with $\bar{y}_{n-1}$, and after this, all possibilities such that $\bar{y}_{n-1}$ shares a site with $\bar{y}_{n-2}$ and so on. Using that $C_{\rho}\varphi(\alpha)<r_0<1$ we compute it to obtain \begin{equation*}
\E[\nu_{\theta}\otimes \nu_{\theta}(D(n))] \leq 2^n \Big[\sum_{m\ge 2} (C_{\rho} \varphi(\alpha))^m \Big] \Big[\sum_{m\ge 2} m(C_{\rho} \varphi(\alpha))^m \Big]^{n-1}.
\end{equation*} 
The choice of $r_0$ implies that $\sum_{m\ge 2} m(C_{\rho} \varphi(\alpha))^m<1/2$ and it follows that \begin{equation*}
\E[\sum\limits_{n\geq 1} \nu_{\theta}\otimes \nu_{\theta}(D(n))] = \sum\limits_{n\geq 1} \E [\nu_{\theta}\otimes \nu_{\theta} (D(n))] < \infty\,.
\end{equation*} Finally, for almost every realization of $\theta$ we have $\nu_{\theta}\otimes \nu_{\theta}(D(n)) \to 0$ when $n\to +\infty$. 
\end{proof}

\begin{lema} \label{nu-times-nu(A)=1} Let $\rho$ and $\alpha$ as in Lemma \ref{no-percolacion-nu-times-nu}. Recall that $A_n$ be the event that exists a separating set $\Delta$ that contains $\Lambda_n$ and $A = \cap_{n\ge 1} A_n$. 

Let $\xi$, $\xi'$ be finite cycle permutations and $\Lambda\Subset \Z^d$. Then for almost every realization of $\theta$ we have $\nu_{\theta,\Lambda}^{\xi} \otimes\nu_{\theta,\Lambda}^{\xi'}(A)=1$.
\end{lema}

\begin{proof}
Define $\Delta_0(\eta,\eta')$ as \[\Delta_0(\eta,\eta')= \sup \{\Delta \Subset \Z^d \colon 0 \in \Delta,\, \Delta \text{ separates } (\eta,\eta')\}\,.\] Such $\Delta_0$ exists because the separating set property is closed by finite unions. 

For the event $\{\Delta_0 = \Z^d\}$ we understand that any site with non-zero multiplicity is in $\Delta_0$. Note that events $A$ and $\{\Delta_0 = \Z^d\}$ are equivalent. In fact, if $A$ does not hold there exists $n$ such that $\Delta$ does not separate $(\eta, \eta')$ for all $\Delta \supset \Lambda_n$. So, $\Delta_0$ cannot contain $\Lambda_n$. Reciprocally, suppose that $x \notin \Delta_0$ with $\theta(x)\neq 0$ and the event $A$ holds. So, there exists a finite set $\Delta$ such that $\Delta$ separates $(\eta, \eta')$ and $\Delta \supset \Lambda_n$ where $n$ is fixed such that $x\in \La_n$. Then $\Delta_0\cup \Delta$ is a separating set that contradicts the maximal assumption of $\Delta_0$.

Assume that exists $x\in\Delta_0^c$ with $\theta(x)\neq 0$. In such case, there exists $s\in \Omega_{\theta}$ with $X(s)=x$ such that $(\eta(s),\eta'(s)) \neq (s,s)$. Otherwise, $\Delta_0\cup \{x\}$ contains $0$ and it is a separating set for $(\eta,\eta')$ larger than $\Delta_0$. Call $\g_1$ to the cycle from $\eta$ or $\eta'$ such that $\g_1(s)\neq s$. Clearly $\g_1$ is open and as $\Delta_0$ is a separating set we have $\{\g_1\}\subset \Delta_0^c$.

Consider $\Delta_1=\Delta_0 \cup \{\g_1\}$. It cannot be a separating set for $(\eta,\eta')$ by the maximal assumption of $\Delta_0$. So, there exists $\g_2$ in $\eta$ or $\eta'$ such that $\Delta_1 \cap \{\g_2\} \neq \emptyset$ and $\Delta_1^c \cap \{\g_2\}\neq \emptyset$. From the second intersection, we deduce $\Delta_0^c \cap \{\g_2\} \neq \emptyset$ but as $\Delta_0$ is a separating set we have $\{\g_2\}\subset \Delta_0^c$. We have also $\{\g_2\}\cap\{\g_1\}^c\neq \emptyset$, so $\g_2\neq\g_1$. From the first intersection we obtain that $\g_2\Join \g_1$, because $\{\g_2\}\cap\{\g_1\}\neq \emptyset$ and $\g_2\neq\g_1$. Note that $\g_2$ is open.

Now, suppose that there are $n$ different open cycles $\g_1, \dots, \g_n$ such that $\{\g_i\}\subset \Delta_0^c$ for all $i=1,\dots,n$ and for each $\g_i$ there exists $j\in \{1,\dots, i-1\}$ such that $\g_j\Join \g_i$. Denote by $\Delta_n$ the set $\Delta_0 \cup \left(\cup_{i=1}^n \{\bar{\g}_i\}\right)$. Since $\Delta_n$ cannot separate $(\eta,\eta')$ there exists a cycle $\g_{n+1}$ in $\eta$ or $\eta'$ (so, $\g_{n+1}$ is open) such that $\Delta_n \cap \{\g_{n+1}\} \neq \emptyset$ and $\Delta_n^c \cap \{\g_{n+1}\} \neq \emptyset$. As $\Delta_0$ is a separating set, the second condition implies that $\{\g_{n+1}\} \subset \Delta_0^c$ and $\g_{n+1}\neq \g_{j}$ for all $j=1,\dots,n$. Then the condition $\Delta_n \cap \{\g_{n+1}\} \neq \emptyset$ says that $\g_{n+1} \Join \g_j$ for some $j=1,\dots, n$.

Thus, there is a sequence of different open cycles $\{\g_i\}_{i\in \mathbb{N}}$ such that each $\g_n$ is the neighbor of some $\g_j$ with $j<n$. Then all cycles from the sequence are in the same connected component, so, it is infinite. Hence, we have proved that $A^c \subset \{\text{exists an infinite open path of cycles containing }x\}\,,$ and by Lemma~\eqref{no-percolacion-nu-times-nu} it follows that $\nu_{\theta,\La}^{\xi}\otimes \nu_{\theta,\La}^{\xi'}(A^c)=0$ for almost every realization of $\theta$.
\end{proof}

\begin{lema} \label{mu-times-mu(A)=1} Let $\rho$ and $\alpha$ as in  Lemma~\ref{no-percolacion-nu-times-nu}. Let $\mu$ and $\mu'$ be Gibbs measures that concentrate on finite cycle permutations. Then $\mu\otimes \mu'(A)=1$.
\end{lema}

\begin{proof} As $A_{n+1}\subset A_{n}$ it is sufficient to show that $\lim_{n\to \infty} \mu\otimes \mu'(A_n)=1$. Using that $A_n$ is a decreasing event with the definition of Gibbs measures, we obtain for $\Lambda\Subset \Z^d$: 
\begin{equation*}
\mu\otimes \mu'(A_n) = \int G_{\theta, \Lambda}^{\xi} \otimes G_{\theta, \Lambda}^{\xi'}(A_n) \dd \mu \otimes \mu'(\xi,\xi') \geq \int \nu_{\theta, \Lambda}^{\xi} \otimes \nu_{\theta, \Lambda}^{\xi'}(A_n) \dd \mu \otimes \mu'(\xi,\xi') \,.
\end{equation*} To complete the proof, we take limit as $n$ tends to $\infty$ in the last term and use the Lemma~\eqref{nu-times-nu(A)=1}.
\end{proof}

\begin{lema} \label{unicidad-discreto} Consider $\rho\in(0,1/2)$ and $\alpha>0$ such that $C_{\rho}\varphi(\alpha)<r_0$, where $r_0$ is the constant defined in~\eqref{r_0}. If $\mu$ and $\mu'$ are Gibbs measures supported on finite cycle permutations, then $\mu=\mu'$.
\end{lema}

\begin{proof} It is sufficient to prove that $\mu(B)=\mu'(B)$ for a local event $B$. 

Let $J_{\La}(\Delta)$ be the event that $\Delta\Subset \Z^d$ is the first separating set that contains $\La$. The existence of $\Delta$ is guaranteed by~\eqref{mu-times-mu(A)=1} and so, $\sum_{\Delta\supset \La, \, \Delta \Subset \Z^d} \mu \otimes \mu'( J_{\La}(\Delta))=1\,.$

Now, suppose that for any finite cycle boundary conditions $\xi$ and $\xi'$, we have \begin{equation} \label{igualdad-especificaciones-producto}
G_{\theta,\Delta}^{\xi}\otimes G_{\theta,\Delta}^{\xi'}\left((B\times S_{\theta}^F)\cap J_{\La}(\Delta)\right) = G_{\theta,\Delta}^{\xi}\otimes G_{\theta,\Delta}^{\xi}\left((S_{\theta}^F \times B)\cap J_{\La}(\Delta) \right)\,.
\end{equation} Then, integrate with respect to $\mu\otimes \mu'$ and use that $\mu$ and $\mu'$ are supported on finite cycle permutations to obtain \[\mu \otimes \mu' \left((B\times S_{\theta}^F)\cap J_{\La}(\Delta)\right) = \mu \otimes \mu'\left((S_{\theta}^F \times B)\cap J_{\La}(\Delta)\right)\,,\] and summing over all choices of $\Delta$, we show $\mu(B) = \mu'(B)$.

So, it remains to prove that (\ref{igualdad-especificaciones-producto}). Observe that if $\Delta$ does not separate $(\xi,\xi')$, both sides of (\ref{igualdad-especificaciones-producto}) are 0. If $\Delta$ separates, the boundary conditions $(\xi,\xi')$ have the same effect as the identity boundary conditions. Hence, we can replace $\xi$ and $\xi'$ by $\text{id}$. Since the event $J_{\La}(\Delta)$ and $G_{\theta,\Delta}^{\text{id}}\otimes G_{\theta,\Delta}^{\text{id}}$ are invariant under the map $(\s,\s') \mapsto (\s',\s)$ the equation (\ref{igualdad-especificaciones-producto}) holds.
\end{proof}


\section{Existence and uniqueness on the continuum} \label{section-continuum}

In this section we study the existence of Gibbs measure when the set of points is a realization of a homogeneous Poisson point process on $\R^d$ with low intensity. The Hamiltonian $H$ is given by the quadratic potential as in~\eqref{Hamiltonian-intro-general}. We understand the finite volume $\Lambda$ as a compact subset of $\R^d$ and we write $\Lambda\Subset \R^d$ for it. The thermodynamic formalism has analogous definitions that in the random lattice case. 

Let $\Omega \subset \R^d$ be the realization of a Poisson point process with density $\rho$. The notation is the same that for previous sections, we write $\Omega$ instead $\Omega_{\theta}$ or $\theta$. So, $S_{\Omega}$ is the permutation space, $S_{\Omega}^F$ is the set of finite cycle permutation and $\Gamma_{\Omega}$ is the space of finite cycles. The support of a cycle has the same definition as before. However, the notion of ordered support does not make sense in the continuous setting, since each jump contributes non-zero to the Hamiltonian.

Section~\eqref{section-domination} works as in the previous case because we did not use anything related to the environment. So, the free process and the loss network of cycles have the same definitions and properties that in the discrete case. In particular, the specification at finite volume $\Lambda$ corresponding to a finite cycle boundary condition is stochastically dominated by the corresponding invariant measure of the free process. So, to show the existence we will prove tightness of the family of specifications $\{G_{\Omega,\Lambda}^{\text{id}}\}_{\Lambda\Subset \R^d}$. To prove uniqueness we will apply the existence of separating sets as in Lemma~\eqref{mu-times-mu(A)=1}. 

We want to construct a coupling between the free process in the continuum setup with the free process of a certain discrete model on $\Z^d$ with Poisson multiplicities in such way that the first is dominated by the second. Then we are able to apply the results of previous sections. 

If $z=(z_1,\dots,z_d) \in \R^d$ we write $\floor{z}=(\floor{z_1},\dots,\floor{z_d})\in \Z^d$. Let $\Omega$ be a homogeneous Poisson process on $\R^d$ with intesity $\rho$. For $x\in \Z^d$, let $\theta(x)$ be the number of points in $\Omega \cap I_x$, where $I_x = x + [0,1)^d$. Then $\theta=\{\theta(x)\}_{x\in \Z^d}$ is an i.i.d. sequence of Poisson($\rho$) random variables. For each $x$ such that $\theta(x) \neq 0$ we tag points of $I_x$ from $1$ to $\theta(x)$ with some rule. For example, using the relative order of the distance to $x$. So, there is a bijection among $\Omega$ and $\Omega_{\theta}$, where $\Omega_{\theta}\subset \Z^d\times \mathbb{N}$ is the set associated to $\theta$, such that each $z\in\Omega$ is mapped to some $(\floor{z},i)$ with $i\in\{1,\dots,\theta(\floor{z})\}$. This bijection induces also a bijective map among the finite cycle spaces. We denote this bijection by $\Psi$. The following lemma summarizes these claims.

\begin{lema} The map $\Psi \colon S_{\Omega} \to S_{\theta}$ is a homeomorphism. Further, it induces a homeomorphism among $\mathbb{N}_0^{\Gamma_{\Omega}}$ and $\mathbb{N}_0^{\Gamma_{\theta}}$ by the relation $\eta \mapsto \varsigma$ where $\varsigma(\g)= \eta(\Psi^{-1}(\g))$ for all $\g\in \Gamma_{\theta}$.

\end{lema}

The following lemma tells us what potential we can choose to compare the discrete and continuum models.

\begin{lema} For $x, z\in\R^d$ we have: $\|x-z\|^2 \geq V(\floor{x}-\floor{z})$, where $V$ is defined by $V(x)=\max\{\|x\|^2-2\sqrt{d}\|x\|,0\}$.
\end{lema}

\begin{proof} For $x\in \R^d$ write $x= \floor{x} + \tilde{x}$ with $\tilde{x}\in [0,1)^d$. So, using Cauchy-Schwarz we have
\[ \|x-z\|^2 \geq \|\floor{x} - \floor{z}\|^2 - 2 \|\floor{x} - \floor{z}\| \|\tilde{x} - \tilde{z}\| \geq \|\floor{x} - \floor{z}\|^2 - 2 \sqrt{d}\|\floor{x} - \floor{z}\|.\qedhere\] 
\end{proof}

For the rest of this section, $V$ will denote the potential defined in the previous Lemma and $H_V$ its associated Hamiltonian. Note that $H(\g)\geq H_V(\Psi(\g))$ for all $\g\in \Gamma_{\Omega}$, so, the respective weights satisfies $w(\g)\leq w_V(\Psi(\g))$ for all $\g \in \Gamma_{\Omega}$. 

Denote by $(\eta_t^o \colon t\in \R)$ and $(\varsigma_t^o \colon t\in \R)$ the stationary constructions of free processes for the continuum model with quadratic potential and for the discrete model related to $V$ respectively. By the relation among weights, we can give a coupling between processes such that $\Psi(\eta_t^o) \leq \varsigma_t^o$ for all $t$. Each free process is constructed as a function of a Poisson process in the cycle spaces, so, it is sufficient to couple both processes. Denote by $\mathcal{N}$ and $\mathcal{N}_V$ the Poisson processes corresponding to $(\eta_t^o \colon t\in \R)$ and $(\varsigma_t^o \colon t\in \R)$ respectively. It is sufficient to construct $\mathcal{N}$ as an independent thinning of $\mathcal{N}_V$ as follows: a mark $(\Psi(\g),t,s) \in\mathcal{N}_V$ induces a mark $(\g,t,s)$ in the process $\mathcal{N}$ with probability $w(\g)/w_V(\Psi(\g))$ independent of each other. 

Let $K\subset S_{\Omega}^F \subset \{0,1\}^{\Gamma_{\Omega}}$ be a decreasing event. Since the Lemma~\eqref{dominacion-de-especificaciones} holds for the continuum setting, the specifications are dominated by the Poisson measure $\nu_{\Omega}$. Using this fact with the coupling among both free processes we have that for all compact set $\Lambda \subset \R^d$: \begin{equation}\label{discreto-domina-continuo}
G_{\Omega,\Lambda}^{\text{id}} (K^c)\leq \nu_{\Omega} (K^c)\leq \nu_{\theta}(\Psi(K^c)).
\end{equation}

For the discrete model we define $\widehat{K}_f = \bigcap\limits_{x\in \Z^d} \widehat{K}_f(x) \subset \mathbb{N}_0^{\Gamma_{\theta}}$ where \[\widehat{K}_{f}(x)= \{\eta\in \mathbb{N}_0^{\Gamma_{\theta}} \colon \forall \,\g\in\eta \text{ such that } x\in \g \text{ we have } H_V(\g) \leq f(x)\} \,.\] It is the same set of the previous section but now is associated to $H_V$. The set $K_f= \widehat{K}_f \cap S_{\theta}^{F}$ is a decreasing event by Lemma~\eqref{lema-Kf-creciente} and if the density $\rho$ is good for $V$, $K_f$ is a non-empty compact set for almost every realization of $\theta$. For the rest of the section we assume that $\rho$ is good for $V$.

Note that $\Psi^{-1}(K_f)$ is also a decreasing event and a non-empty compact set. Hence, equation~\eqref{discreto-domina-continuo} implies that for almost every realization of $\Omega$ we have \[G_{\Omega,\Lambda}^{\text{id}} (\Psi^{-1}(K_f^c))\leq \nu_{\Omega} (\Psi^{-1}(K_f^c))\leq \nu_{\theta}(K_f^c)\,.\] So, given $\epsilon>0$, Lemma~\eqref{lema-cota-quenched} provides a function $f$ such that $\nu_{\theta}(K_f^c)<\epsilon$.

\begin{lema} Let $\rho \in (0,1/2)$ such that $\rho$ is good for $V$. Suppose that $\rho$ and $\alpha>0$ satisfy $C_{\rho}\varphi_V(\alpha/2)<1$, where $C_{\rho}= \frac{\rho e^{-\rho+\frac{1}{2}}}{1-2\rho}$ and $\varphi_V$ was defined in~\eqref{funcion-varphi-general}.

Then the family $\{ G_{\Omega,\Lambda}^{\text{id}} \}_{\Lambda \Subset \R^d}$ is tight and there exists a Gibbs measure $\mu$ that concentrates on finite cycle permutations.

\end{lema}

Now, we want to establish the uniqueness among Gibbs measures that concentrate on finite cycle permutations. To prove it, we use the existence of arbitrary large separating sets for the discrete model with potential $V$. 

We say that a compact set $\Delta \subset \R^d$ is a separating set for $\eta\in \mathbb{N}_0^{\Gamma_{\Omega}}$ if $\partial \Delta \cap \Omega = \emptyset$ and for any $\g\in\eta$ we have $\{\g\}\subset \Delta$ or $\{\g\}\subset \Delta^c$. We say that $\Delta$ is a separating set for the pair $(\eta,\eta')$, if it is a separating set for $\eta$ and $\eta'$.

Given $z_1,\dots, z_n \in \Omega$, we can pick a compact set $\Lambda\subset \R^d$ such that $\partial \Lambda \cap \Omega = \emptyset$ and $\Lambda \cap \Omega = \{z_1,\dots, z_n\}$. There exists a lot of ways to choose $\Lambda$, but fix one of these. Now, given $\Delta \Subset \Z^d$ define $\Psi^{-1}(\Delta)$ as the fix previous compact set that contains $\cup_{x\in \Delta} (\Omega \cap I_x)$. Note that, if $\Delta \Subset \Z^d$ is a separating set for $\varsigma \in \mathbb{N}_0^{\Gamma_{\theta}}$, the compact set $\Psi^{-1}(\Delta) \subset \R^d$ is also a separating set for $\eta=\Psi^{-1}(\varsigma)\in \mathbb{N}_0^{\Gamma_{\Omega}}$. This fact can be extended to pair of configurations. 

Let $A$ be the event that exists a sequence of compact sets that increase to $\R^d$ and each set is a separating set for pairs $(\eta,\eta')\in \NN_0^{\Gamma_{\Omega}}\times\NN_0^{\Gamma_{\Omega}}$. By the coupling between continuum and discrete free processes and Lemma~\eqref{nu-times-nu(A)=1} about the existence or arbitrary large separating sets for the discrete model, one shows that for almost surely with respect to the environment $\nu_{\Omega, \Lambda}^{\xi} \otimes \nu_{\Omega, \Lambda}^{\xi'}(A)=1$ for any finite $\Lambda$ and any pair $(\xi, \xi')$ of finite cycle permutations in $\cS_{\Omega}$.

Lemma~\eqref{mu-times-mu(A)=1} also holds in the continuum context, since it only uses the definition of Gibbs measures and the domination of specifications by the corresponding free process. So, if $\mu$ and $\mu'$ are Gibbs measures in the continuous model that concentrate on the finite cycle permutations, we have \begin{equation*} \mu\otimes\mu'( \exists \{\Delta_j\}_{j\in \NN} \uparrow \R^d \text{ such that } \Delta_j \text{is a separating set}) = 1.
\end{equation*} Hence, the existence of an increasing sequence of separating sets in the continuum model is proved.

\begin{lema} Let $\rho \in (0,1/2)$ such that $\rho$ is good for $V$. Suppose that $\rho$ and $\alpha$ satisfy $C_{\rho}\varphi_V(\alpha)<r_0$, where $r_0$ is the solution of the equation~\eqref{r_0}. 

Then, if $\mu$ and $\mu'$ are Gibbs measures supported on finite cycle permutations we have $\mu=\mu'$.
\end{lema}

\begin{proof} The proof of Lemma~\eqref{unicidad-discreto} applies without changes.
\end{proof}

\section*{Appendix}

In this appendix we prove a bound for the number of cycles that have the same ordered support. Recall that the ordered support of a cycle $\g$, defined in~\eqref{ord-supp}, is the projection of $\g$ to $\Z^d$ erasing consecutive repetitions of sites. 

Let $\bar{y}\in (\Z^d)^m$ such that $\bar{y}_i\neq \bar{y}_{i+1}$. It is a possible ordered support. Remember that $N_{\theta}(\bar{y})$ is the number of cycles such that its ordered support is $\bar{y}$. Write $\{\bar{y}\}$ for coordinates of $\bar{y}$ without repetitions. For $z\in \{\bar{y}\}$ let $k_z(\bar{y})= \# \{i \colon y_i=z\}$ be the number of times that $z$ appears in the ordered support $\bar{y}$, i.e., the effective uses of site $z$ in the sense that computes non trivially for the Hamiltonian. If $\bar{y}$ is such that $k_z>\theta(z)$ for some $z\in\{\bar{y}\}$, then the number of cycles that has $\bar{y}$ as ordered support is zero. It is the same if $\theta(z)=0$ for any $z\in \{\bar{y}\}$. Therefore, suppose that $\bar{y}$ is such that $0<k_z\leq \theta(z)$ for all $z\in \{\bar{y}\}$. 

A cycle $\g$ with $[\g]=\bar{y}$ can use the site $z$ at least $k_z$ times and at most $\theta(z)$ times. Let $a\in\{k_z,\dots,\theta(z)\}$ be the number of times that $\g$ uses $z$. There are $\binom{\theta(z)}{a} a!$ ways to choose $a$ different points in $\Omega_{\theta}$ located at $z$. Then, the $a$ different points have to locate at $k_z$ coordinates of $\bar{y}=[\g]$ corresponding to place $z$, and they must use all of $k_z$ coordinates. Thus, there are $\binom{a-1}{k_z-1}$ ways to do it (it is the same problem that put $a$ balls in $k_z$ boxes using at least one ball per box). So, for the number of cycles that has $\bar{y}$ as ordered support we have 
\begin{equation} \label{cotaM(y)} N_{\theta}(\bar{y}) \leq \sum\limits_{\substack{k_1 \leq a_1 \leq \theta(z_1)\\
 \vdots \\
 k_{l} \leq a_{l} \leq \theta(z_{l})}} \prod_{j=1}^{l} \binom{\theta(z_j)}{a_j} a_j! \binom{a_j-1}{k_j-1} \leq \prod_{j=1}^{l} \left( \frac{e^{\frac{1}{2}}}{2} \, \theta(z_j)! \, 2^{\theta(z_j)} \one_{\{\theta(z_j)\neq 0\}} \right) := M_{\theta}(\bar{y})\,,
\end{equation} 
where $\{\bar{y}\}=\{z_1,\dots,z_l\}$ and $k_j=k_{z_j}$. 
 
For density $\rho\in(0,1/2)$, the expectation under $\P$ of the upper bound $M_{\theta}(\bar{y})$ given in~\eqref{cota-expectation-M(y)}, can be computed using the independence of multiplicities and their distribution. Indeed, 
\begin{align*}
\E[M_{\theta}(\bar{y})] = \E\left[\frac{e^{\frac{1}{2}}}{2} \theta(z)!\,2^{\theta(z)} \one_{\{\theta(z)\neq 0 \}} \right]^{l} = \left(\sum\limits_{i\geq 1} \frac{e^{\frac{1}{2}}}{2} 2^{i} \, e^{-\rho}\, \rho^i \right)^l \leq \left(\frac{\rho e^{-\rho+\frac{1}{2}}}{1-2\rho}\right)^{|\bar{y}|} \,.
\end{align*}

\begin{obs} \label{cota-indep-orden} The upper bound $M_{\theta}(\bar{y})$ does not depend on the relative order for coordinates of $\bar{y}$. So, if we need an upper bound for the number of pairs of cycles $(\g,\g')$ such that $[\g]= \bar{y}$ and $[\g']=\bar{y}'$, we use the upper bound $M_{\theta}(\bar{y}\bar{y}')$, where $\bar{y}\bar{y}'$ is a concatenation of vectors such that $(\bar{y},\bar{y}')$ is an ordered support. 
\end{obs}



\bibliographystyle{alpha}

\end{document}